\documentclass[10pt]{article}
\usepackage{amsmath}
\usepackage{amsfonts}
\usepackage{amsthm}
\usepackage{amssymb}
\usepackage{color}

\newtheorem{theorem}{Theorem}[section]

\newtheorem{lemma}[theorem]{Lemma}
\newtheorem{remark}[theorem]{Remark}

\numberwithin{equation}{section}



\title{Nontrival solution for nonlinear $p(x)$--Laplacian Dirichlet problem with the sign--changing weight}

\usepackage{authblk}

\author{Sylwia Dudek}

\affil{\small 
Krakow University of Technology,
Institute of Mathematics,\newline
ul. Warszawska 24, 31-155 Krakow, Poland \newline e--mail: sbarnas@pk.edu.pl}

\date{}

\begin{document}

\maketitle
\bibliographystyle{plain}

\noindent \textbf{Abstract:}
  In this paper we study the nonlinear elliptic problem involving
  $p(x)$--Laplacian with nonsmooth potential, where the weighted function $\lambda$ 
	may change sign. By using critical point theory for locally Lipschitz functionals due to Chang \cite{changg}, 
	we obtain conditions which ensure the existence of a solution for our problem.


\section{Introduction} \label{Intr}
  Let $\Omega\subseteq\mathbb{R}^N$ be a bounded domain with the smooth boundary $\partial \Omega$.
  In this paper we study the following nonlinear hemivariational inequality with $p(x)$--Laplacian

\begin{equation}\label{eq_01}
  \left\{
  \begin{array}{lr}
    -\Delta_{p(x)}u(x)-\lambda |u(x)|^{p(x)-2} u(x) \in \partial j(x, u(x))& \textrm{ a.e. in  } \Omega,\\
    u=0 & \textrm{on}\ \partial \Omega,
  \end{array}
  \right.
\end{equation}
  where $p:\overline{\Omega} \rightarrow \mathbb{R}$ is a continuous function satisfying
\[
  1<p^- \leqslant p(x) \leqslant p^+ < \widehat{p}^* \quad \mbox{for a.e. } \; x\in \Omega
\]
  with $p^-:=\inf\limits_{x \in \overline{\Omega}} p(x)$,  $p^+:=\sup\limits_{x \in \overline{\Omega}} p(x)$ and \[
  \widehat{p}^*:=\left\{
  \begin{array}{ll}
   \frac{Np^-}{N-p^-} & p(x)<N\\
    \infty & p(x) \geqslant N.
  \end{array}
  \right.
\]
The operator 
$
\Delta_{p(x)}u(x):= \textrm{div} \big( |\nabla u(x)|^{p(x)-2} \nabla u(x) \big)
$
  is the so--called $p(x)$--Laplacian.
 The function
$j(x,t)$ is locally Lipschitz
  in the $t$--variable
  and measurable in the $x$--variable and by $\partial j(x,t)$ we denote the subdifferential with respect
  to the $t$--variable in the sense of Clarke \cite{Clarke}.

  Recently, hemivariational inequalities have attracted more and more attention.
  The study of such problems arises in nonlinear elasticity theory and in
  physical phenomena, in which we dealt with nonconvex and nonsmooth energy 
  functionals. We can find such functions for example in fluid mechanics, in the image 
  restoration and in the calculus of variations. Moreover, we deal with the variable exponent spaces.
	The typical examples of equations stated in the variable exponent spaces are models of electrorheological fluids.
	This kind of~materials have been intensively investigated recently.
Electrorheological fluids change their mechanical properties dramatically when 
an external electric field is applied, so the variable exponent settings are natural 
for their modelling. Several applications in the 
  electrorheological fluids problems involving $p(x)$--growth conditions can be
  found in the books of Naniewicz--Panagiotopulous \cite{nana} and 
  Ru$\check{\textrm{z}}$i$\check{\textrm{c}}$ka \cite{ruzicka}.

  The starting point for hemivariational inequalities with $p(x)$--Laplacian were this 
  with constant exponent, it means with $p(x) \equiv p$. For example, the following 
  differential inclusion problem was considered

\begin{equation}\label{eq_04}
  \left\{
  \begin{array}{lr}
    -\Delta_{p}u(x)-\lambda |u(x)|^{p-2} u(x) \in \partial j(x, u(x))& \textrm{a.e. in  } \Omega,\\
    u=0 & \textrm{on}\ \partial \Omega,
  \end{array}
  \right.
\end{equation}
  where $\lambda>0$ is a first eigenvalue of $p$--Laplacian. For instance, the 
  existence of nontrival solution for Dirichlet problem (\ref{eq_04}) at resonance
  under different type of conditions was proved in papers of Gasi\'nski--Papageorgiou
  \cite{lg, lg1, lg2}. Their methods are based on the critical point theory for 
  Locally Lipschitz functionals and on the Ekeland variational principle. 
  Marano--Bisci--Motreanu in \cite{2} proved the existence of multiple solutions for     
  (\ref{eq_04}) by the use of Struwe techniques and the saddle point theory. There are also
  many others authors who studied hemivariational inequalities with Dirichlet or 
  Neumann boundary conditions.

  Partial differential equations involving variable exponents and nonstandard growth
  conditions were also studied by many authors. In the paper of Ge--Xue--Zhou 
  \cite{ge2} the existence of radial solutions for problem (\ref{eq_01}) was proved.
  The authors required that $\lambda>0$ and used a key assumption on the exponent 
  that $p^+<N$. The problem with $p(x)$--Laplacian and with Neumann boundary condition
  was considered by Qian--Shen--Yang \cite{qian}. They refused the assumption about
  positivity of $\lambda$ but still needed assumption on the variable exponent, it 
  means $\sqrt{2}p^->N$.

  In this paper we have the situation that $\lambda \in \mathbb{R}$ and we have no
  restriction on $\lambda$ like in Barna\'s \cite{barnas, barnas2,barnas3} and the 
  papers of many authors. It is an extension of the theory considered in the above 
  mentioned papers. Moreover, we abandon the restriction on the exponent $p(x)$. Our approach is
   based on the critical point theory for nonsmooth Lipschitz functionals due to 
  Chang \cite{changg}.

 In the next section we briefly present the basic properties of the generalized 
 Lebesgue spaces and the generalized Lebesgue--Sobolev spaces.
 Moreover, we present the basic notions and facts from the theory, which will be used
 in the study of problem (\ref{eq_01}).

\section{Mathematical preliminaries} \label{Prelim}

 In order to discuss problem (\ref{eq_01}), we need to state some properties of the 
 spaces $L^{p(x)}(\Omega)$ and $W^{1,p(x)}(\Omega)$,
 which we call correspondingly generalized Lebesgue spaces and generalized Lebesgue--Sobolev spaces (see Fan--Zhao \cite{fan,orlicz}).

Denote by
$
  E(\Omega)
$ the set of all measurable real functions defined on $\Omega$.
  Two functions in $E(\Omega)$ are considered to be one element
  of $E(\Omega)$, when they are equal almost everywhere. The generalized Lebesgue space is defined as
\[
    L^{p(x)}(\Omega)=\{u \in E(\Omega): \int_{\Omega} |u(x)|^{p(x)} dx<\infty \},
\]
  equipped with the norm
\[
    \|u\|_{p(x)}=\|u\|_{L^{p(x)}(\Omega)}=\inf \Big\{\lambda>0: \int_{\Omega} \Big|\frac{u(x)}{\lambda}\Big|^{p(x)}dx \leqslant 1 \Big\}.
\]
  Next, we define the generalized Lebesgue--Sobolev space $W^{1,p(x)}(\Omega)$ by
\[
    W^{1,p(x)}(\Omega) = \{u \in L^{p(x)}(\Omega):\nabla u \in L^{p(x)}(\Omega; \mathbb{R}^N)\}
\]
  with the norm
\[
    \|u\|=\|u\|_{W^{1,p(x)}(\Omega)} = \|u\|_{p(x)}+\|\nabla u\|_{p(x)}.
\]

  Then $(L^{p(x)}(\Omega),\|\cdot\|_{p(x)})$ and $(W^{1,p(x)}(\Omega),\|\cdot\|)$ are
  separable and refelxive Banach spaces. By $W_0^{1,p(x)}(\Omega)$ we denote the 
  closure of $C_0^\infty (\Omega)$ in $W^{1,p(x)}(\Omega)$.

\begin{lemma} [Fan--Zhao \cite{fan}]  \label{fan}
 If $\Omega \subseteq \mathbb{R}^N$ is an open domain, then

\noindent 
  (a) if $1 \leqslant q(x) \in \mathcal{C}(\overline{\Omega})$ and $q(x) \leqslant p^*(x)$ (respectively $q(x) < p^*(x)$) for any $x \in \overline{\Omega}$, where
\[
  p^*(x)=\left\{
  \begin{array}{ll}
   \frac{Np(x)}{N-p(x)} & p(x)<N\\
    \infty & p(x) \geqslant N,
  \end{array}
  \right.
\]
  then $W^{1,p(x)}(\Omega)$ is embedded continuously (respectively compactly) in $L^{q(x)}(\Omega)$;

\noindent
 (b) Poincar\'e inequality in $W_0^{1,p(x)}(\Omega)$ holds i.e., there exists a positive constant $c$ such that
 \[
  \|u\|_{p(x)} \leqslant c \|\nabla u\|_{p(x)}  \qquad \textrm{for all } u \in W_0^{1,p(x)}(\Omega);
\]

\noindent
 (c) $(L^{p(x)}(\Omega))^*=L^{p'(x)}(\Omega)$, where $\frac{1}{p(x)}+\frac{1}{p'(x)}=1$
   and for all $u \in L^{p(x)}(\Omega)$ and $v \in L^{p'(x)}(\Omega)$, we have
\[
  \int_{\Omega}|uv|dx\leqslant \Big(\frac{1}{p^-}+\frac{1}{
  {p'}^{-}}\Big)\|u\|_{p(x)}\|v\|_{p'(x)}.
\]
\end{lemma}

\begin{lemma}[Fan--Zhao \cite{fan}]\label{lemma2}
  Let $\varphi (u)=\int_{\Omega} |u(x)|^{p(x)} dx$ for $u \in L^{p(x)}(\Omega)$ and let
  $\{u_n\}_{n \geqslant 1} \subseteq L^{p(x)}(\Omega)$.

\noindent (a) for $a\neq 0$, we have \quad
 $\|u\|_{p(x)} =a \Longleftrightarrow \varphi (\frac{u}{a})=1$;

\noindent (b) we have
\begin{center}$\|u\|_{p(x)}<1 \; \Longleftrightarrow \; \varphi(u)<1$;\end{center}

\begin{center}$\|u\|_{p(x)}=1 \; \Longleftrightarrow \;  \varphi(u)=1$;\end{center}

\begin{center}$\|u\|_{p(x)}>1 \; \Longleftrightarrow \; \varphi(u)>1$;\end{center}

\noindent (c) if $\|u\|_{p(x)}>1$, then
\begin{center}$\|u\|^{p^-}_{p(x)} \leqslant \varphi (u) \leqslant \|u\|^{p^+}_{p(x)}$;\end{center}

\noindent (d) if $\|u\|_{p(x)}<1$, then
\begin{center}$\|u\|^{p^+}_{p(x)} \leqslant \varphi (u) \leqslant \|u\|^{p^-}_{p(x)}$;\end{center}

\noindent (e) we have
\begin{center}$\lim\limits_{n \rightarrow \infty} \|u_n\|_{p(x)} =0 \; \Longleftrightarrow \; \lim\limits_{n \rightarrow \infty} \varphi(u_n) =0$;\end{center}

\noindent (f) we have
\begin{center}$\lim\limits_{n \rightarrow \infty}\|u_n\|_{p(x)} = \infty \; \Longleftrightarrow \; \lim\limits_{n \rightarrow \infty} \varphi (u_n) =\infty$.\end{center}

\end{lemma}

  Similarly to Lemma \ref{lemma2}, we have the following result.

\begin{lemma} [Fan--Zhao \cite{fan}] \label{lemma6}
  Let $\Phi (u) = \int_{\Omega} (|\nabla u(x)|^{p(x)}+|u(x)|^{p(x)})dx$
  for $u \in W^{1,p(x)}(\Omega)$ and let $\{u_n\}_{n \geqslant 1} \subseteq W^{1,p(x)}(\Omega)$.
  Then

\noindent (a) for $a \neq 0$, we have
\begin{center}$\|u\|=a \; \Longleftrightarrow \; \Phi (\frac{u}{a})=1$;\end{center}

\noindent (b) we have
\begin{center}$\|u\|<1 \; \Longleftrightarrow \; \Phi (u) <1$;\end{center}

\begin{center}$\|u\|=1 \; \Longleftrightarrow \; \Phi (u) =1$;\end{center}

\begin{center}$\|u\|>1 \; \Longleftrightarrow \; \Phi (u)>1$;\end{center}

\noindent (c) if $\|u\|>1$, then
\begin{center} $\|u\|^{p^-} \leqslant \Phi (u) \leqslant  \|u\|^{p^+}$;\end{center}

\noindent (d) if $\|u\|<1$, then
\begin{center}$\|u\|^{p^+} \leqslant \Phi (u) \leqslant  \|u\|^{p^-}$;\end{center}

\noindent (e) we have
\begin{center}$\lim\limits_{n \rightarrow \infty} \|u_n\|=0 \Longleftrightarrow \lim\limits_{n \rightarrow \infty} \Phi (u_n) =0$;\end{center}

\noindent (f) we have
\begin{center}$\lim\limits_{n \rightarrow \infty} \|u_n\| = \infty \Longleftrightarrow \lim\limits_{n \rightarrow \infty} \Phi (u_n) = \infty$.\end{center}

\end{lemma}

  Consider the following function
\[
    J(u)=\int_{\Omega} \frac{1}{p(x)}|\nabla u|^{p(x)} dx, \qquad \textrm{for all } u \in W_0^{1,p(x)}(\Omega).
\]
  We know that $J \in \mathcal{C}^1 (W_0^{1,p(x)}(\Omega))$ and $-\textrm{div}(|\nabla u|^{p(x)-2} \nabla u)$
  is the derivative operator of $J$ in the weak sense (see Chang \cite{chang}).
  We denote
\[
  A=J': W_0^{1,p(x)}(\Omega) \rightarrow (W_0^{1,p(x)}(\Omega))^*,
\]
  then
\begin{equation} \label{operator}
  \langle Au,v\rangle = \int_{\Omega} |\nabla u(x)|^{p(x)-2} (\nabla u(x), \nabla v(x)) dx
  \end{equation}
for all  $u, v \in W_0^{1,p(x)}(\Omega).$

\begin{lemma}[Fan--Zhang \cite{zhang}] \label{lemma3}
  If A is the operator defined above, then $A$ is a continuous,
  bounded and strictly monotone operator of type $(S)_+$ i.e.,\\
   $u_n \rightarrow u$ weakly in $W_0^{1,p(x)}(\Omega)$ and
  $\limsup\limits_{n \rightarrow \infty} \langle Au_n, u_n -u \rangle \leqslant 0$
  implies that $u_n \rightarrow u$ in  $W_0^{1,p(x)}(\Omega)$.
\end{lemma}

  Let ($X$, $\| \cdot \|$) be a Banach space and $X^*$ its topological dual.
A function $f:X\rightarrow\mathbb{R}$ is said to be locally Lipschitz,
  if for every $x\in X$ there exists a neighbourhood  $U$ of $x$
  and a constant $K>0$ depending on $U$ such that
  $|f(y)-f(z)| \leqslant K\|y-z\|$ for all $y,z \in U$.
  From convex analysis it is well know that a proper, convex and lower
  semicontinuous function
  $g: X \rightarrow \overline{\mathbb{R}}=\mathbb{R}\cup \{ + \infty\}$
  is locally Lipschitz in the interior of its domain $\textrm{dom} g=\{x \in X:g(x)< \infty\}$.

 In analogy with the directional derivative of a convex function,
  we introduce the notion of the generalized directional derivative of a locally Lipschitz function $f$ at $x \in X$ in the direction $h \in X$
  by
\[
    f^0 (x;h)=\limsup_{y \rightarrow x, \lambda \searrow 0} \frac{f(y+\lambda h)-f(y)}{\lambda}.
\]
  The function $h \longmapsto f^0 (x,h) \in \mathbb{R}$ is sublinear and 
  continuous so it is the support function of a nonempty, $w^*$--compact and convex set
\[
    \partial f(x)=\{x^* \in X^*: \langle x^*,h \rangle \leqslant f^0 (x,h) \textrm{ for all } h \in X\}.
\]
  The set $\partial f(x)$ is known as generalized or Clarke subdifferential of $f$ at
  $x$. If $f$ is convex, then the subdifferential in the sense of convex analysis 
  coincides with the generalized subdifferential introduced above. 

  The critical point theory for smooth functions uses a compactness condition known 
  as "Cerami condition" (C--condition for short). In our present nonsmooth 
  settings, the condition takes the following form.
	
	\medskip

    We say that $f$ satisfies the "nonsmooth Cerami condition" (nonsmooth C--condition for short),
  if any sequence $\{x_n\}_{n \geqslant 1} \subseteq X$ such that $\{f(x_n)\}_{n \geqslant 1}$
  is bounded and $(1+\|x_n\|)m(x_n)\rightarrow 0$ as $n \rightarrow \infty$, where $m(x_n)=\min\{ \|x^*\|_* :x^* \in \partial f(x_n)\}$, has a strongly convergent subsequence.
	
	\medskip

  The first theorem is due to Chang \cite{changg} and extends to
  a nonsmooth setting the well known mountain pass theorem due
  to Ambrosetti--Rabinowitz \cite{ambro}.

\begin{theorem} \label{twierdzenie}
  If $X$ is a reflexive Banach space, $R:X \rightarrow \mathbb{R}$ is
  a locally Lipschitz functional satisfying C--condition and for some
  $\rho>0$ and $y \in X$ such that $\|y\|>\rho$, we have
\[
    \max\{R(0),R(y)\}<\inf\limits_{\|x\|=\rho} \{R(x)\}=: \eta,
\]
   then R has a nontrivial critical point $x \in X$ such that
   the critical value $c=R(x) \geqslant \eta$ is characterized by the following minimax principle
\[
    c=\inf\limits_{\gamma \in \Gamma}\max\limits_{0 \leqslant \tau \leqslant 1}\{R(\gamma(\tau))\},
\]
  where $\Gamma=\{\gamma \in \mathcal{C}([0,1],X):\gamma(0)=0,\gamma(1)=y\}$.
\end{theorem}

  The second theory is an other nonsmooth version of mountain pass theorem.

\begin{theorem}\label{twierdzenie2}
  If $X$ is a reflexive Banach space and $R: X \rightarrow \mathbb{R}$ is a bounded 
  below and locally Lipschitz functional which satisfies nonsmooth C--condition, then
  $c=\inf \{R(x):x \in X \}$ is a critical value of $R$.
\end{theorem}

\section{Existence of Solutions}
  We start by introducing our assumptions for the nonsmooth potential $j(x,t)$.\\

\noindent $H(j)$ $\; j:\Omega  \times \mathbb{R} \rightarrow \mathbb{R}$ is a function such that $j(x,0)=0$ a.e. in $\Omega$ and

\medskip

\noindent \textbf{(i)} for all $t \in \mathbb{R}$, the function $\Omega \ni x \rightarrow j(x,t) \in \mathbb{R}$ is measurable;

\smallskip

\noindent \textbf{(ii)} for almost all $x \in \Omega$, the function $\mathbb{R} \ni t \rightarrow j(x,t) \in \mathbb{R}$ is locally Lipschitz;

\smallskip

\noindent \textbf{(iii)} for almost all $x \in \Omega$ and all $v \in \partial j(x,t)$, we have $|v| \leqslant c_1|t|^{r(x)-1}$ with $r \in \mathcal{C}(\overline{\Omega})$
such that $p^+ < r^-:=\min\limits_{x \in \overline{\Omega}} r(x) \leqslant r(x) <\widehat{p}^*$ and $c_1>0$;

\smallskip

\noindent \textbf{(iv)} there exists $c>2 c_1$ such that
\[ \label{war}
  \limsup\limits_{|t|\rightarrow \infty} \frac{v^*t-j(x,t)}{|t|^{r(x)}} \leqslant -c,
\]
  uniformly for almost all $x \in \Omega$ and all $v^*\in \partial j(x, t)$.
	
	\bigskip

  We introduce locally Lipschitz functional $R: W_0^{1,p(x)}(\Omega) \rightarrow \mathbb{R}$ defined by
\[
  R(u)=\int_{\Omega} \frac{1}{p(x)}|\nabla u(x)|^{p(x)}dx -\int_{\Omega}  
  \frac{\lambda}{p(x)}|u(x)|^{p(x)}dx - \int_{\Omega} j(x,u(x))dx,
\]
  for all $u \in W_0^{1,p(x)}(\Omega)$.

\begin{lemma}\label{PS}
If hypothesis $H(j)$ hold, then $R$ satisfies the nonsmooth C--condition.
\end{lemma}

\begin{proof}

  Let $\{u_n\}_{n \geqslant 1} \subseteq W_0^{1,p(x)}(\Omega)$ be a sequence such that 
  $\{R(u_n)\}_{n \geqslant 1}$ is bounded and $m(u_n) \rightarrow 0$ as $n \rightarrow \infty.$
  We will show that $\{u_n\}_{n \geqslant 1} \subseteq W_0^{1,p(x)}(\Omega)$ is bounded.

  Because $|R(u_n)|\leqslant M$ for all $n \geqslant 1$, we have
\begin{equation} \label{1111}
  -M \leqslant \int_{\Omega} \frac{1}{p(x)} |\nabla u_n (x)|^{p(x)} dx-\int_{\Omega} 
  \frac{\lambda}{p(x)}|u_n (x)|^{p(x)}dx - \int_{\Omega} j(x,u_n (x))dx.
\end{equation}

  Since $\partial R(u_n) \subseteq (W_0^{1,p(x)}(\Omega))^*$ is weakly compact, 
  nonempty and the norm functional is weakly lower semicontinuous in a Banach space,
  then we can find $u_n^* \in \partial R(u_n)$ such that $\|u_n^*\|_*=m(u_n)$ for $n \geqslant 1$.

  Consider the operator $A:W_0^{1,p(x)}(\Omega) \rightarrow (W_0^{1,p(x)}(\Omega))^*$
  defined by (\ref{operator}). Then, for every $n \geqslant 1$, we have
\begin{equation}\label{11ss}
  u_n^*=Au_n-\lambda |u_n|^{p(x)-2} u_n - v_n^*,
\end{equation}
  where $v_n^* \in \partial \psi (u_n)\subseteq L^{p'(x)} (\Omega)$, for $n \geqslant 1$,
  with $\frac{1}{p(x)}+\frac{1}{p'(x)}=1$ and $\psi: W_0^{1,p(x)}(\Omega) \rightarrow \mathbb{R}$ is defined by $\psi (u_n)=\int\limits_{\Omega} j(x,u_n(x)) dx$. We know 
  that, if $v_n^* \in \partial \psi (u_n)$, then $v_n^*(x) \in \partial j(x, u_n(x))$ (see Clarke \cite{Clarke}). 

  From the choice of the sequence $\{u_n^*\}_{n \geqslant 1} \subseteq W_0^{1,p(x)}(\Omega)$, at least for a subsequence, we have
\begin{equation}\label{syl}
  |\langle u_n^*,w \rangle| \leqslant \frac{\varepsilon_n\|w\|}{1 +\|u_n\|}
	\quad \textrm{for all } w \in
  W^{1,p(x)}_0(\Omega),
\end{equation}
  with $\varepsilon_n \searrow 0$.
  Putting $w=u_n$ in (\ref{syl}) and using (\ref{11ss}), we obtain
\begin{equation} \label{236k}
  -\varepsilon_n \leqslant - \int_{\Omega} |\nabla u_n(x)|^{p(x)} dx +\lambda 
  \int_{\Omega} |u_n(x)|^{p(x)} dx +\int_{\Omega} v_n^* (x)u_n(x) dx.
\end{equation}

\noindent Now, let us consider two cases.\\

\noindent \textit{Case $1$. }
 Let $\lambda \leqslant 0$.

  Adding (\ref{1111}) and (\ref{236k}), we have
\begin{eqnarray} \label{568}
  -M-\varepsilon_n &\leqslant&  \Big( \frac{1}{p^-}-1 \Big) \int_{\Omega}|\nabla u_n 
  (x)|^{p(x)} dx+ |\lambda| \Big( \frac{1}{p^-}-1\Big)\int_{\Omega} |u_n 
  (x)|^{p(x)}dx \nonumber\\
                   &         &\qquad \qquad +\int_{\Omega} v_n^* (x)u_n(x) dx - \int_{\Omega} j(x,u_n (x))dx.
\end{eqnarray}
  So we obtain that
\begin{eqnarray} \label{335}
   |\lambda|\Big( 1-\frac{1}{p^-}\Big) \int_{\Omega} |u_n (x)|^{p(x)}dx 
  \leqslant \quad \quad \nonumber \\
  M+\varepsilon_n +\int_{\Omega} v_n^* (x)u_n(x) dx - \int_{\Omega} j(x,u_n (x))dx.
\end{eqnarray}

  By virtue of hypotheses $H(j)(iv)$, we know that there exist constant $c>2 c_1$, such that  
\[
  \limsup\limits_{|t|\rightarrow \infty} \frac{v^*t-j(x,t)}{|t|^{r(x)}} \leqslant -c,
\]
  uniformly for almost all $x \in \Omega$ and all $v^* \in \partial j(x,t)$ with $p^+ < r^- \leqslant r(x) < \widehat{p}^*$ for all $x \in \Omega$. So in particularly, 
  there exists $L>0$ such that for almost all $x \in \Omega$ and all $|t| \geqslant L$, we have
\begin{equation} \label{waznee}
  v^*t-j(x,t) \leqslant -\frac{c}{2}|t|^{r(x)}.
\end{equation}
  On the other hand, from the Lebourg mean value theorem (see Clarke \cite{Clarke}),
  for almost all $x \in \Omega$ and all $t \in \mathbb{R}$, we  can find $v(x) \in \partial j(x, k u(x))$ with $0<k<1$, such that
\[
  |j(x,t)-j(x,0)| \leqslant |v||t|.
\]
  So from hypothesis $H(j)(iii)$, for almost all $x \in \Omega$, we have
\[ \label{210}
  |j(x,t)| \leqslant |j(x,0)|+c_1|t|^{r(x)} \leqslant c_1|t|^{r^+}.
\]
  Then for almost all $x \in \Omega$ and all $t$ such that $|t|<L$, it follows that
\begin{equation} \label{123}
  |j(x,t)| \leqslant c_2,
\end{equation}
  for some $c_2>0$. Therefore, from (\ref{waznee}) and (\ref{123}) it follows that 
  for almost all $x \in \Omega$ and all $ t \in \mathbb{R}$, we have
\begin{equation} \label{2045}
  v^*t-j(x,t) \leqslant -\frac{c}{2}|t|^{r(x)}+\beta,
\end{equation}
  for some $\beta>0$ and $p^+ < r^- \leqslant r(x) < \widehat{p}^*$ for all $x \in \Omega$.

  We use (\ref{2045}) in (\ref{335}) and obtain
\[
 |\lambda|\Big( 1-\frac{1}{p^-}\Big) \int_{\Omega} |u_n (x)|^{p(x)}dx 
  \leqslant M+\varepsilon_n - \frac{c}{2}\int_{\Omega}|u_n(x)|^{r(x)} dx +  
  \int_{\Omega} \beta dx,
\]
  which leads to
\[
  |\lambda| \Big( 1-\frac{1}{p^-} \Big)\int_{\Omega} |u_n (x)|^{p(x)}dx 
\leqslant M_1,
\]
  for some $M_1>0$.
  We know that $ |\lambda| \Big( 1-\frac{1}{p^-} \Big) >0$, so
\begin{equation}\label{ogr}
  \textrm{the sequence } \{u_n\}_{n \geqslant 1} \subseteq L^{p(x)} (\Omega) \textrm{ is bounded.}
\end{equation}

  Now, consider again (\ref{568}) to obtain
\[
  \Big( 1-\frac{1}{p^-}\Big) \int_{\Omega} |\nabla u_n (x)|^{p(x)}dx 
  \leqslant M+\varepsilon_n +\int_{\Omega} v_n^* (x)u_n(x) dx - \int_{\Omega} j(x,u_n (x))dx.
\]
  In a similar way, by using (\ref{2045}) we have
\[
   \Big( 1-\frac{1}{p^-}\Big) \int_{\Omega} |\nabla u_n (x)|^{p(x)}dx 
  \leqslant  M+\varepsilon_n - \frac{c}{2} \int_{\Omega} |u_n(x)|^{r(x)} dx+\int_{\Omega} \beta dx,
\]
  for all $n \geqslant 1$ with $p^+ < r^- \leqslant r(x) < \widehat{p}^*$ for all $x \in \Omega$.
Hence, we get
\[
  \Big( 1-\frac{1}{p^-}\Big) \int_{\Omega} |\nabla u_n (x)|^{p(x)}dx 
\leqslant M_2,
\]
  for some $M_2>0$.
  So, we have that
\begin{equation}\label{ogr1}
  \textrm{the sequence } \{\nabla u_n\}_{n \geqslant 1} \subseteq L^{p(x)} (\Omega;\mathbb{R}^N) \textrm{ is bounded.}
\end{equation}

  From (\ref{ogr}) and (\ref{ogr1}), we have that
\[
  \textrm{the sequence } \{u_n\}_{n \geqslant 1} \subseteq W_0^{1,p(x)} (\Omega) \textrm{ is bounded.}
\]

\noindent \textit{Case $2$. }
Now, let $\lambda > 0$. 

  Again from (\ref{1111}) and (\ref{236k}), we have
\begin{eqnarray} \label{890}
  -M-\varepsilon_n &\leqslant&  \Big( \frac{1}{p^-}-1 \Big) \int_{\Omega}|\nabla u_n 
   (x)|^{p(x)} dx+ \lambda \Big(1- \frac{1}{p^+}\Big)\int_{\Omega} |u_n (x)|^{p(x)}dx \nonumber\\
                   &         &\qquad \qquad +\int_{\Omega} v_n^* (x)u_n(x) dx - \int_{\Omega} j(x,u_n (x))dx.
\end{eqnarray}
  Since $\Big(\frac{1}{p^-}-1\Big)<0$ and using (\ref{2045}), we have 
\begin{eqnarray*} 
  -M-\varepsilon_n &\leqslant& \lambda \Big(1- \frac{1}{p^+}\Big)\int_{\Omega} |u_n
  (x)|^{p(x)}dx \nonumber\\
                   &         &\qquad \qquad -\frac{c}{2} \int_{\Omega} |u_n(x)|^{r(x)} dx + \int_{\Omega}\beta dx.
\end{eqnarray*}
  for all $n \geqslant 1$ and $p^+ < r^- \leqslant r(x) < \widehat{p}^*$ for all $x \in \Omega$.
  Hence, we have
\begin{eqnarray*} 
\frac{c}{2} \int_{\Omega} |u_n(x)|^{r(x)} dx &\leqslant& \lambda \Big(1- \frac{1}{p^+}\Big)\int_{\Omega} |u_n (x)|^{p(x)}dx +K,
\end{eqnarray*}
  for some $K>0$.
 Since $ p(x) \leqslant p^+ <r^- \leqslant r(x)$  for all $x \in \Omega$, we have
\[
  \textrm{the sequence } \{u_n\}_{n \geqslant 1} \subseteq L^{r(x)} (\Omega) \textrm{ is bounded.}
\]

  For any $n \geqslant 1$ such that $\|u_n\|_{p(x)} \leqslant 1$ we have 
\[
  \|u_n\|_{p(x)}^{p^+}< \int_{\Omega} |u_n(x)|^{p(x)} dx < \int_{\Omega} |u_n(x)|^{p^-} dx \leqslant K_1,
\]
  for some $K_1>0$ (see Lemma \ref{lemma2}).

  On the other hand, for any $n \geqslant 1$ such that $\|u_n\|_{p(x)} > 1$, we have
\[
  \|u_n\|_{p(x)}^{p^-}< \int_{\Omega} |u_n(x)|^{p(x)} dx < \int_{\Omega} 
  |u_n(x)|^{p^+} dx < \|u_n\|_{r(x)}^{r(x)} \leqslant K_2,
\]
with some $K_2>0$.
  Thus
\begin{equation} \label{ss}
  \textrm{the sequence } \{u_n\}_{n \geqslant 1} \subseteq L^{p(x)} (\Omega) \textrm{ is bounded}.
\end{equation}

  Now, again from (\ref{890}), we have
\begin{eqnarray}\label{338899} 
  \Big( 1-\frac{1}{p^-} \Big) \int_{\Omega}|\nabla u_n (x)|^{p(x)} dx \leqslant
   M+\varepsilon_n+ \lambda \Big(1- \frac{1}{p^+}\Big)\int_{\Omega} |u_n 
  (x)|^{p(x)}dx \nonumber \nonumber\\
  +\int_{\Omega} v_n^* (x)u_n(x) dx - \int_{\Omega} j(x,u_n (x))dx.
\end{eqnarray}
  Using (\ref{2045}) and (\ref{ss}) in (\ref{338899}), we obtain
\[
  \Big( 1-\frac{1}{p^-} \Big) \int_{\Omega}|\nabla u_n (x)|^{p(x)} dx \leqslant  M_3,
\]
  for some $M_3>0$.
  Since $\Big( 1-\frac{1}{p^-}\Big)>0$, we have that
\begin{equation}\label{ogr3}
\textrm{the sequence } \{\nabla u_n\}_{n \geqslant 1} \subseteq L^{p(x)} (\Omega;\mathbb{R}^N) \textrm{ is bounded.}
\end{equation}
  From (\ref{ss}) and (\ref{ogr3}), we have that
\[
  \textrm{the sequence } \{u_n\}_{n \geqslant 1} \subseteq W_0^{1,p(x)} (\Omega) \textrm{ is bounded.}
\]
 From Cases $1$ and $2$, we have that 
\[
  \textrm{the sequence } \{u_n\}_{n \geqslant 1} \subseteq W_0^{1,p(x)} (\Omega) \textrm{ is bounded}.
\]

  Hence, by passing to a subsequence if necessary, we may assume that
\begin{equation}\label{1}
  \left.
  \begin{array}{ll}
   u_n \rightarrow u & \textrm{weakly in } W_0^{1,p(x)} (\Omega),\\
   u_n \rightarrow u & \textrm{in } L^{p(x)}(\Omega),
  \end{array}
  \right.
\end{equation}
  for some $u \in  W_0^{1,p(x)} (\Omega)$.
  Putting $w=u_n-u$ in (\ref{syl}) and using (\ref{11ss}), we obtain
\begin{eqnarray} \label{47}
  \Big|\langle Au_n,u_n-u\rangle - \lambda \int_{\Omega} 
  |u_n(x)|^{p(x)-2}u_n(x)(u_n-u)(x)dx\nonumber\\
	-\int_{\Omega} v_n^*(x) (u_n-u)(x)dx\Big|\leqslant \varepsilon_n,
\end{eqnarray}
  with $\varepsilon_n \searrow 0$.
  Using Lemma \ref{fan}(c), we see that
\begin{eqnarray*}
  &           & \lambda \int_{\Omega} |u_n(x)|^{p(x)-2}u_n(x)(u_n-u)(x)dx\cr
  & \leqslant & \lambda \Big(\frac{1}{p^-}+\frac{1}{p'^-}\Big) \|\, |u_n|^{p(x)-1}\|_{p'(x)}\|u_n-u\|_{p(x)},
\end{eqnarray*}
   where $\frac{1}{p(x)}+\frac{1}{p'(x)}=1$.
  We know that the sequence $\{u_n\}_{n \geqslant 1} \subseteq L^{p(x)}(\Omega)$
  is bounded, so using (\ref{1}), we can conclude that
\[
  \lambda \int_{\Omega} |u_n(x)|^{p(x)-2}u_n(x)(u_n-u)(x)dx \rightarrow 0 \quad \textrm{as } n \rightarrow \infty
\]
  and
\[
  \int_{\Omega} v_n^*(x) (u_n-u)(x)dx \rightarrow 0 \quad \textrm{as } n \rightarrow \infty.
\]
  If we pass to the limit as $n \rightarrow \infty$ in (\ref{47}), we have
\[
  \limsup\limits_{n \rightarrow \infty} \langle Au_n, u_n-u \rangle \leqslant 0.
\]
  So from Lemma \ref{lemma3}, we have that $u_n \rightarrow u$ in $W^{1,p(x)}_0(\Omega)$ as $ n \rightarrow \infty$. Thus $R$ satisfies the  C--condition.
\end{proof}

\bigskip

  For the first existence theorem, we will need an additional assumption\\

\noindent \textbf{$H(j)_1$} there exists $\nu>0$ such that
\[
  \limsup\limits_{|t|\rightarrow 0} \frac{j(x,t)}{|t|^{h(x)}} \leqslant -\nu,
\]
  uniformly for almost all $x \in \Omega$ and for some $h(x) \in \mathcal{C} (\overline{\Omega})$ with $1<h(x) \leqslant h^+< p^-<\widehat{p}^*$ for all $x \in \Omega$.\\

\begin{theorem}
 If hypotheses $H(j)$ and $H(j)_1$ hold
then problem (\ref{eq_01}) has a nontrival solution for all $\lambda \in (-\infty, \nu p^-)$.
\end{theorem}

\begin{proof}
 \textbf{Claim.1.} There exists $\rho \in (0, 1)$ small enough such that, we have 
  $R(u)\geqslant L$, for all $u \in W_0^{1,p(x)}(\Omega)$ with $\|u\|=\rho$ and some 
  $L>0$.\\
	Indeed by using hypothesis $H(j)_1$, we can find $\delta>0$, such that for
  almost all $x \in \Omega$ and all $t$ such that $|t|\leqslant \delta$, we have
\[
  j(x,t) \leqslant -\nu|t|^{h(x)}, \quad \textrm{ where } \quad 1<h(x) \leqslant h^+< p^-.
\]
  On the other hand, from hypothesis $H(j)(iii)$, we know
  that for almost all $x \in \Omega$ and all $t$ such that $|t|>\delta$, we have
\[
 |j(x,t)| \leqslant c_1|t|^{r(x)},
\]
  where $p^+ <r(x)<\widehat{p}^*$ for all $x \in \Omega$.
  Thus for almost all $x \in \Omega$ and all $t \in \mathbb{R}$ we have
\begin{equation}\label{porow}
  j(x,t) \leqslant-\nu|t|^{h(x)}+d_1 |t|^{r(x)},
\end{equation}
  with some $d_1>0$, $1<h(x) \leqslant h^+< p^- \leqslant p(x) \leqslant p^+ < r^-\leqslant r(x)<\widehat{p}^*$ for all $x \in \Omega$.

  Moreover, since $W_0^{1,p(x)}(\Omega)$ is embedded continuously into $L^{p(x)}(\Omega)$, $L^{h(x)}(\Omega)$ and $L^{r(x)}(\Omega)$ 
	(see Lemma (\ref{fan})), so for $\beta(x):= p(x)$ (respectively $h(x)$ or $r(x)$), we have that 

\begin{equation}\label{wlozenie1}
  \|u\|_{\beta(x)} \leqslant K_3 \|u\|, 
\end{equation}
  for all $u \in W_0^{1,p(x)}(\Omega)$ and some $K_3>0$.

  If we fix $\rho \in (0,1)$ such that $\rho < \min \{ 1, \frac{1}{K_3} \}$, then for
  all $u \in W_0^{1,p(x)}(\Omega)$, with $\|u\|=\rho$, from (\ref{wlozenie1}) we can deduce that 
\[
 \|u\|_{\beta(x)} \leqslant 1 \quad \textrm{ where } \quad \beta(x):=p(x) \textrm{ (respectively } h(x) \textrm{ or } r(x)).
\]
  Futhermore, using Lemma \ref{lemma2} and (\ref{wlozenie1}), we obtain

\begin{equation} \label{nier1}
  \int_{\Omega} |u(x)|^{\beta(x)} dx \leqslant \|u(x)\|^{\beta^-}_{\beta(x)} \leqslant K_3 \|u\|^{\beta^-},
\end{equation}
  for $\beta(x):=p(x)$ (respectively $h(x)$ or $r(x)$).

  Moreover, since $1<h(x) \leqslant h^+<p(x)\leqslant p^+ < r^-\leqslant r(x)$, then for all $u \in W_0^{1,p(x)}(\Omega)$, with $\|u\|=\rho$, we have that 
\begin{equation}\label{crucial}
  \|u\|_{r(x)} \leqslant\|u\|_{p(x)} \leqslant \|u\|_{h(x)}.
\end{equation}

  Let us consider two cases.

\bigskip

\noindent \textit{Case 1. }
 Let $\lambda \leqslant 0$. 

  By using (\ref{porow}), (\ref{nier1}) and Lemma \ref{lemma2}, we obtain that
\begin{eqnarray*}
   R(u) &    =    &\int_{\Omega} \frac{1}{p(x)} |\nabla u (x)|^{p(x)} dx-\int_{\Omega} \frac{\lambda}{p(x)}|u (x)|^{p(x)}dx - \int_{\Omega} j(x,u (x))dx\\
        &\geqslant&  \frac{1}{p^+} \int_{\Omega} |\nabla u (x)|^{p(x)} dx + \frac{|\lambda|}{p^+} \int_{\Omega} |u (x)|^{p(x)}dx\\
        &         & \qquad \qquad \qquad +\nu\int_{\Omega} |u(x)|^{h(x)} dx - d_1 \int_{\Omega}|u(x)|^{r(x)} dx\\
        &\geqslant& c_5 \|u\|^{p^+} - d_1 \int_{\Omega}|u(x)|^{r(x)} dx \geqslant c_5 \|u\|^{p^+}-d_1\|u\|^{r^-},
\end{eqnarray*}
  where $c_5=\min \{ \frac{1}{p^+}, \frac{|\lambda|}{p^+} \}$ and $\nu>0$.

  Since $p^+<r^-\leqslant r(x)$ for all $x \in \Omega$, we have $R(u)\geqslant L>0$, for all $u \in W_0^{1,p(x)}(\Omega)$, with $\|u\|=\rho$.

\bigskip

\noindent \textit{Case 2. }
 Let $\lambda >0$.

  Using (\ref{porow}) and (\ref{crucial}), we obtain that
\begin{eqnarray*}
   R(u) &     =     &\int_{\Omega} \frac{1}{p(x)} |\nabla u (x)|^{p(x)} dx-\int_{\Omega} \frac{\lambda}{p(x)}|u (x)|^{p(x)}dx - \int_{\Omega} j(x,u (x))dx\\
        & \geqslant &  \frac{1}{p^+} \int_{\Omega} |\nabla u (x)|^{p(x)} dx-\frac{\lambda}{p^-} \int_{\Omega}|u(x)|^{p(x)}dx\\
        &           &\hspace{2.7cm}+ \nu\int_{\Omega}|u(x)|^{h(x)} dx - d_1 \int_{\Omega}|u(x)|^{r(x)} dx\\
        &\geqslant  & \frac{1}{p^+} \int_{\Omega} |\nabla u(x)|^{p(x)} dx + \Big(\nu-\frac{\lambda}{p^-}\Big) \int_{\Omega} |u(x)|^{p(x)}dx-d_1 \int_{\Omega}|u(x)|^{r(x)} dx.\\
\end{eqnarray*}
  From hypothesis, we know that $\nu-\frac{\lambda}{p^-} >0$ and by using (\ref{nier1}), we have
\[
  R(u) \geqslant c_6 \|u\|^{p^+}-d_1 \|u\|^{r^-},
\]
  where $c_6=\min \{ \frac{1}{p^+}, \nu-\frac{\lambda}{p^-}\}$.

  So again, we have that $R(u)\geqslant L>0$, for all $u \in W_0^{1,p(x)}(\Omega)$, with $\|u\|=\rho$.\\

\textbf{Claim.2.} 
  $R(u)$ is anticoercive, i.e. $R(u) \rightarrow -\infty$ as $\|u\| \rightarrow \infty$. \\
  We assume that $\|u\|>1$. 
  Again using hypothesis $H(j)(iv)$, for almost all $x \in \Omega$ and all $t$ such that $t>M$, we have

\begin{equation} \label{2041}
  j(x,t) \geqslant v^*t +\frac{c}{2}|t|^{r(x)}-\beta,
\end{equation}
  for some $\beta>0$ and $p^+ < r^- \leqslant r(x) < \widehat{p}^*$ for all $x \in \Omega$ (see (\ref{2045})).

  On the other hand, from $H(j)(iii)$, we see that for almost all $x \in \Omega$ we
  have $|v^* t| \leqslant c_1|t|^{r(x)}$, where $c_1>0.$ So from  (\ref{2041}) and 
  this inequality, we obtain
\begin{equation}\label{xx}
  j(x,t) > \frac{c}{2} |t|^{r(x)} - c_1 |t|^{r(x)}-\beta=c_3|t|^{r(x)}-\beta,
\end{equation}
  where $c_3>0$ (since $c>2 c_1$) with $p(x) \leqslant p^+ <r(x) \leqslant r^+< \widehat{p}^*$.

  Using (\ref{xx}) and Lemma \ref{lemma2}, for any  $u \in W^{1,p(x)}_0 (\Omega) \backslash \{0\}$ and $s>1$, we have
\begin{eqnarray*}
  R(su)&    =     & \int_{\Omega} \frac{1}{p(x)} |\nabla s u (x)|^{p(x)} dx-\int_{\Omega} \frac{\lambda}{p(x)}|s u (x)|^{p(x)}dx - \int_{\Omega} j(x,s u(x))dx\\
       &\leqslant & s^{p^+} \Big(\frac{1}{p^-} \int_{\Omega} | \nabla u (x)|^{p(x)} dx +\frac{|\lambda|}{p^{-}} \int_{\Omega} |u(x)|^{p(x)} dx \Big)-\int_{\Omega} j(x,s u(x))dx\\
       &\leqslant & \overline{c} \cdot s^{p^+} \big(\int_{\Omega} (|\nabla u (x)|^{p(x)}+|u(x)|^{p(x)})dx\big) - c_3 \int_{\Omega} |s u(x)|^{r(x)}dx+\int_{\Omega} \beta dx\\
       &\leqslant & \overline{c} \cdot s^{p^+}\|u\|^{p^+}- c_3\cdot s^{r^-} \int_{\Omega} |u(x)|^{r(x)}dx+\int_{\Omega} \beta dx,
\end{eqnarray*}
  where $\overline{c}=\max\{ \frac{1}{p^-},\frac{|\lambda|}{p^-}\}$ and $p^+ <r^+\leqslant r(x) <\widehat{p}^*.$\\

  Because $r^->p^+$, we get that
  $R(s u) \rightarrow -\infty$ when $s \rightarrow \infty$. This permits the use of  
  Theorem \ref{twierdzenie} which gives us $u \in W_0^{1,p(x)}(\Omega)$ such that 
  $R(u)>0 = R(0)$ and $0 \in \partial R(u)$.

From the last inclusion we obtain
\[
  0=Au-\lambda |u|^{p(x)-2}u-v^*,
\]
  where $v^* \in \partial \psi(u).$ Hence
\[
  Au=\lambda |u|^{p(x)-2}u+v^*,
\]
  so for all $v \in \mathcal{C}_0^\infty (\Omega)$, we have $\langle Au,v \rangle = \lambda \langle |u|^{p(x)-2}u,v\rangle +\langle v^*,v \rangle$. 

  So we have
\begin{eqnarray*}
  &   & \int_{\Omega} |\nabla u(x)|^{p(x)-2}(\nabla u(x), \nabla v(x))_{\mathbb{R}^N}dx\cr
  & = & \int_{\Omega}\lambda |u(x)|^{p(x)-2}u(x)v(x)dx+\int_{\Omega}v^*(x)v(x) dx,
\end{eqnarray*}
  for all $v \in \mathcal{C}_0^\infty(\Omega)$.

  From the definition of the distributional derivative we have
\[
  \left\{
  \begin{array}{lr}
    -\textrm{div} \big( |\nabla u(x)|^{p(x)-2} \nabla u(x) \big)=\lambda |u(x)|^{p(x)-2} u(x) +v(x)& \textrm{in } \Omega,\\
    u=0 & \textrm{on}\ \partial \Omega,
  \end{array}
  \right.
\]
  so
\[
  \left\{
  \begin{array}{lr}
    -\Delta_{p(x)}u(x)-\lambda |u(x)|^{p(x)-2} u(x)\in \partial j(x, u(x))& \textrm{in } \Omega,\\
    u=0 & \textrm{on}\ \partial \Omega.
  \end{array}
  \right.
\]
  Therefore $u \in W_0^{1,p(x)}(\Omega)$ is a nontrivial solution of (\ref{eq_01}).
\end{proof}

\begin{remark}
  A nonsmooth potential satisfying hypotheses $H(j)$ and $H(j)_1$ is for example the one given by the following function
\[
  j_1(x,t)= 
   \left\{
     \begin{array}{lcc}
      -\nu|t|^{h(x)} & \textrm{if} & |t| \leqslant 1,\\
      -|t|^{r^+}-\nu+1 & \textrm{if} & |t| > 1,\\
     \end{array}
    \right.
\]
  with $\nu >0$ and continuous functions $h, r:\overline{\Omega} \rightarrow \mathbb{R}$ which satisfy  $1< h(x) \leqslant h^+<p^- \leqslant p(x) \leqslant p^+<r^- \leqslant r(x) \leqslant r^+<\widehat{p}^* $.\\
\end{remark}

\bigskip

  Instead of hypothesis $H(j)_1$ we can take additional assumption about behaviour 
  in infinity and also obtain existence of a nontrival solution.\\

\noindent \textbf{$H(j)_2$} there exists $\mu> 2c_1$ such that
\[
  \limsup\limits_{|t|\rightarrow \infty} \frac{j(x,t)}{|t|^{r(x)}} \leqslant -\mu,
\]
  uniformly for almost all $x \in \Omega$ with $1<p(x) \leqslant p^+< r^-\leqslant r(x) <\widehat{p}^*$ for all $x \in \Omega$. \\

\begin{theorem}
 If hypotheses $H(j)$ and $H(j)_2$ hold 
 then problem (\ref{eq_01}) has a nontrival solution for any $\lambda \in \mathbb{R}$.
\end{theorem}

\begin{proof}
  We claim that $R(u)$ is bounded below. We assume that $\|u\|>1$.

  By virtue of hypotheses $H(j)_2$, we know that there exist constants $\mu>2c_1$ and
  $L>0$ such that for almost all $x \in \Omega$ and all $|t| \geqslant L$, we have
\begin{equation} \label{waznee1}
  j(x,t) \leqslant -\frac{\mu}{2}|t|^{r(x)}.
\end{equation}

  On the other hand, from the hypothesis $H(j)(iii)$, for almost all $x \in \Omega$ and all $t<L$, we have
\begin{equation}\label{21011}
  |j(x,t)| \leqslant  c_1|t|^{r(x)},
\end{equation}
  with $p^+<r^- \leqslant r(x)$.
  Therefore, from (\ref{waznee1}) and (\ref{21011}) it follows that for almost all $x \in \Omega$ and all $ t \in \mathbb{R}$, we have
\[
  j(x,t) \leqslant (c_1-\frac{\mu}{2})|t|^{r(x)} \leqslant -k|t|^{r(x)},
\]
  for some $k>0$ (since $\mu>2c_1$) and $p^+ < r^- \leqslant r(x) < \widehat{p}^*$ for all $x \in \Omega$.

Hence, we have
\begin{eqnarray*}
   R(u)&    =   &\int_{\Omega} \frac{1}{p(x)} |\nabla u (x)|^{p(x)} dx-\int_{\Omega} \frac{\lambda}{p(x)}|u (x)|^{p(x)}dx - \int_{\Omega} j(x,u (x))dx\\
      &\geqslant&  \frac{1}{p^+} \int_{\Omega} |\nabla u (x)|^{p(x)} dx-\frac{\lambda_+}{p^-} \int_{\Omega}|u(x)|^{p(x)}dx+ k \int_{\Omega}|u(x)|^{r(x)} dx\\
      &\geqslant& k \int_{\Omega}|u(x)|^{r(x)} dx -\frac{\lambda_+}{p^-} \int_{\Omega}|u(x)|^{p(x)}dx,\\
\end{eqnarray*}
  where $\lambda_+ :=\max \{0, \lambda \}.$
  Since $r^->p^+$, so $R(u)>L>0$ for all $u \in W^{1,p(x)}_0(\Omega)$ with $\|u\|>1$.\\

  We know that $R$ satisfies C--condition. So we apply Theorem \ref{twierdzenie2} and
  obtain $u_0 \in W^{1,p(x)}_0(\Omega)$ such that $R(u_0)=\inf \{R(u): u \in W^{1,p(x)}_0 (\Omega) \}$. This implies that $u_0$ is a critical point of $R$, and so it is a solution of (\ref{eq_01}).
\end{proof}

\begin{remark}
  The existence of a nontrival solution for problem (\ref{eq_01}) was also considered in the
  papers of Barna\'s \cite{barnas,barnas2,barnas3}. In contrast to the last papers, we have no restriction on $\lambda$, it means $\lambda \in \mathbb{R}$.   
  Moreover, we make the hypothesis as simple as possibe. In hypothesis $H(j)(iv)$, we assume a Tang--type condition which is more general than Landesman--Lazer or Ambrosetti--Rabinowitz condition.
\end{remark}

\end{document}